\documentclass[11pt]{amsart}

\usepackage{amsmath,amssymb,esint,color}
\usepackage{amsthm}
\usepackage{hyperref}
\usepackage[margin=1.0in]{geometry}

\numberwithin{equation}{section}

\newtheorem{prop}{Proposition}

\newtheorem{thm}[prop]{Theorem}

\numberwithin{prop}{section}

\theoremstyle{definition}

\newtheorem{rmk}[prop]{Remark}

\newcommand{\delb}{\bar{\partial}}\newcommand{\dt}{\frac{\partial}{\partial t}}
\newcommand{\brs}[1]{\left| #1 \right|}

\newcommand{\gG}{\Gamma}

\newcommand{\gD}{\Delta}

\newcommand{\gw}{\omega}

\renewcommand{\ge}{\epsilon}
\newcommand{\N}{\nabla}
\newcommand{\FF}{\mathcal F}

\newcommand{\LL}{\mathcal L}

\renewcommand{\bar}[1]{\overline{#1}}

\DeclareMathOperator{\Rc}{Rc}

\DeclareMathOperator{\tr}{tr}

\begin{document}

\title[K\"ahler stability of symplectic forms]{K\"ahler stability of symplectic forms}

\begin{abstract}
\end{abstract}

\author{Jeffrey Streets}
\address{Rowland Hall\\
        University of California, Irvine\\
        Irvine, CA 92617}
\email{\href{mailto:jstreets@uci.edu}{jstreets@uci.edu}}

\author{Gang Tian}
\address{BICMR and SMS, Peking University\\ 
Beijing, P.R. China, 100871}
\email{gtian@math.pku.edu.cn}

\date{\today}

\begin{abstract} Using dynamical stability of symplectic curvature flow, we show that on a compact Calabi-Yau manifold, any small symplectic deformation of a K\"ahler form remains K\"ahler.
\end{abstract}

\thanks{The first author acknowledges support from the NSF via DMS-1454854. The second author is partially supported by NSFC-11890661.  This article is warmly dedicated to Peter Li on the occasion of his 70th birthday.}

\maketitle

\section{Introduction}

A central theme in complex and symplectic geometry is to understand the stability of various properties under natural deformations, and the relationship to uniqueness and moduli problems.  For instance, the classic result of Kodaira-Spencer \cite{KodairaSpencer} shows that small complex deformations of K\"ahler manifolds remain K\"ahler.  The main result of this note is local stability of K\"ahler structures under symplectic deformations on Calabi-Yau manifolds:

\begin{thm} \label{t:mainthm} Let $(M^{2n}, \gw, J)$ be a compact K\"ahler manifold with $c_1(M, J) = 0$.  There exists $\ge > 0$ depending on $\gw$ so that if $\gw'$ is another symplectic form on $M$ such that
\begin{align*}
\brs{\gw - \gw'}_{C^{\infty}(\gw)} < \ge,
\end{align*}
then there exists an integrable complex structure $J'$ on $M$ compatible with $\gw'$.
\end{thm}

\begin{rmk}
\begin{enumerate}
\item The proof of Theorem \ref{t:mainthm} is an elementary consequence of the dynamical stability of symplectic curvature flow (SCF) \cite{SCF} near Calabi-Yau metrics.  This dynamical stability of the more general family of `almost Hermitian curvature flows' introduced in \cite{SCF} was shown in the thesis of Smith (cf. \cite{SmithSCF}).  We sketch the proof below in the simplified case of SCF.
\item It follows from the proof that in fact the $C^{\infty}$ smallness of the perturbation can be weakened to smallness in an appropriate Sobolev space.
\item The space of deformations of symplectic cohomology classes which remain K\"ahler under deformation was studied in \cite{deBart}.
\item Theorem \ref{t:mainthm} follows in some cases using results from complex deformation theory, and our proof provides an alternative using a geometric flow adapted to almost K\"ahler geometry.
\item Recently, the result of Theorem \ref{t:mainthm} was shown in the case $n = 3$ in \cite{IIAflow} using a geometric flow of symplectic forms adapted to that dimension.
\end{enumerate}
\end{rmk}

To begin we recall fundamental properties of symplectic curvature flow \cite{SCF}.  An almost K\"ahler structure is a pair $(\gw, J)$ of a symplectic structure together with a compatible almost complex structure $J$ such that $g = \gw J$ is a Riemannian metric.  In general $J$ is not integrable and $N$ will denote the Nijenhuis tensor of $J$.  Almost K\"ahler structures come equipped with a Chern connection, the unique connection $\N$ on the tangent bundle such that $\N g \equiv 0, \N J \equiv 0$, and $T^{1,1} = 0$, where $T^{1,1}$ denotes the $(1,1)$ component of the torsion of $\N$.  Let $\Omega$ denote the curvature of $\N$, and define $P = \tr \Omega J \in \pi c_1(M, J)$.  A one-parameter family of almost K\"ahler structures $(g_t, \gw_t, J_t)$ satisfies symplectic curvature flow if
\begin{gather} \label{f:SCF}
\begin{split}
\dt g =&\ -2 \Rc + \frac{1}{2} B^1 - B^2,\\
\dt \gw =&\ - P,\\
\dt J =&\ - D^* D J + \mathcal N + \mathcal R,
\end{split}
\end{gather}
where $D$ denotes the Levi-Civita connection, and
\begin{align*}
B^1_{ij} =&\ g^{kl} g_{mn} D_i J_k^m D_j J_l^n, \qquad \qquad B^2_{ij} = g^{kl} g_{mn} D_k J_i^m D_l J_j^n,\\
\mathcal N_i^j =&\ g^{jk} g_{mn} g^{pq} D_p J_r^m J_i^r D_q J_k^n, \qquad \mathcal R_i^j = J_i^k \Rc_k^j - \Rc_i^k J_k^i.
\end{align*}
Note that the our description of symplectic curvature flow is redundant as any two of $(g, \gw, J)$ suffices to recover the third by compatibility.  The fundamental points (cf. \cite{SCF} Theorem 1.6) are that symplectic curvature flow is locally well-posed for arbitrary initial data on compact manifolds, preserves the almost K\"ahler conditions, and if $J_0$ is integrable reduces to K\"ahler-Ricci flow.

\begin{thm} \label{t:dynstabSCF} (cf. \cite{SmithSCF} Theorem 1.1) Let $(M^{2n}, \gw_{CY}, J_{CY})$ denote a compact Calabi-Yau manifold.  There exists $\ge > 0$ so that if $(\gw, J)$ is an almost K\"ahler structure such that
\begin{align*}
\brs{\gw_{CY} - \gw}_{C^{\infty}(\gw_{CY})} + \brs{J_{CY} - J}_{C^{\infty}(\gw_{CY})} < \ge,
\end{align*}
then the solution to symplectic curvature flow with initial condition $(\gw, J)$ exists on $[0,\infty)$ and converges exponentially to a K\"ahler Calabi-Yau structure $(\gw_{\infty}, J_{\infty})$.
\end{thm}

\begin{proof}
The proof relies on ideas from parabolic regularity theory and so we work directly with the gauge-modified flow which is strictly parabolic.  For any almost K\"ahler structure $(g, J)$ we define the vector field
\begin{align*}
X^k(g, J) = g^{ij} \left( \gG_{ij}^k - (\gG_{CY})_{ij}^k \right).
\end{align*}
An elementary but important point is that this vector field is equivalently expressed as
\begin{align*}
X^k =&\ \gw^{ij} \N^{CY}_i J_j^k.
\end{align*}
Using $X$ we define the gauge-fixed symplectic curvature flow:
\begin{gather} \label{f:gfSCF}
\begin{split}
\dt g =&\ -2 \Rc + \frac{1}{2} B^1 - B^2 + L_X g =: \FF_1(g,J)\\
\dt J =&\ - D^* D J + \mathcal N + \mathcal R + L_X J =: \FF_2(g,J)
\end{split}
\end{gather}
The analysis centers on a sharp characterization of the linearization of this flow.  To find this fix a one-parameter family of almost K\"ahler structures $(g_t, J_t)$ such that $(g_0, J_0) = (g_{CY}, J_{CY})$ and $\dot{g} = h, \dot{J} = K$.  Lengthy but straightforward computations using that $(g_0, J_0)$ is Calabi-Yau show (cf. \cite{SCF} proof of Theorem 1.6)
\begin{align*}
\mathcal L_{(g_{CY}, J_{CY})} \mathcal \FF_1 (h,K) =&\ \gD h + 2 R \circ h =: \mathcal L_1(h),\\
\mathcal L_{(g_{CY}, J_{CY})} \mathcal \FF_2 (h,K) =&\ \gD K + 2 R \circ K =: \mathcal L_2(K),\\
\end{align*}
where
\begin{align*}
(R \circ h)_{ij} =&\ R_{i k l j} h^{kl}, \qquad (R \circ K)_j^i = g^{kl} R_{j l m}^i J_k^m.
\end{align*}
To analyze this operator we recall the work of Koiso \cite{Koisodef}.  The operator $\LL_1$ is the Einstein deformation operator at a Ricci-flat metric, and splits according to the decomposition $h = h_S + h_A$ into the  $J$-symmetric and $J$-antisymmetric pieces.  The action on $h_S$ corresponds precisely to the Hodge Laplacian acting on the $(1,1)$-form $h_S J_{CY}$, which is negative semidefinite with kernel determined by harmonic $(1,1)$ forms, which are canonically identified with $H^{1,1}(M, J_{CY})$.  The action on $h_A$ is identified, after raising an index with $g_{CY}$, with the $\delb$-Hodge Laplacian acting on $\Lambda^{0,1} \otimes T^{1,0}$, in this case restricted to symmetric endomorphisms.  Furthermore, the operator $\mathcal L_2$ is again this same $\delb$-Hodge Laplacian acting on $\Lambda^{0,1} \otimes T^{1,0}$, whose kernel is identified with the space of deformations of $J_{CY}$.  This again is negative semidefinite with kernel identified with $H^{2,0}(M, \mathbb C)$.

Thus we have shown that the linearized operator is negative semidefinite, with kernel identified with the space of Einstein deformations of the given Calabi-Yau.  It follows from a result in \cite{Tiandeformation} that every such infinitesimal deformation is in fact integrable.
Given this weak linear stability, together with an explicit description of the kernel, which is integrable, the remainder of the proof follows standard lines (cf. for instance \cite{Lotaystab, Sesum, SmithSCF,HCF}).  In particular, by treating the flow as a small perturbation of the linearized flow, and using the analysis of the linearized operator above, one can show exponential decay towards some Calabi-Yau structure.  Given this exponential convergence, it is elementary to show that the family of diffeomorphisms relating (\ref{f:gfSCF}) and (\ref{f:SCF}) converges exponentially, and thus the solution to (\ref{f:SCF}) is also converging to a Calabi-Yau structure exponentially fast.
\end{proof}

We now prove Theorem \ref{t:mainthm} as a consequence of Theorem \ref{t:dynstabSCF}:

\begin{proof}[Proof of Theorem \ref{t:mainthm}] Given $(M^{2n}, \gw, J)$ a compact K\"ahler manifold with $c_1(M, J) = 0$, by Yau's theorem \cite{YauCC} there exists a unique Calabi-Yau metric $\gw_{CY} \in [\gw]$ compatible with $J$.  Applying Moser's Lemma to the family of cohomologous symplectic forms $\gw_t = t \gw_{CY} + (1-t) \gw$ we obtain the existence of a diffeomorphism $\phi$ such that $\phi^* \gw_{CY} = \gw$.  By construction the pair $(\gw, \phi^* J) = (\phi^* \gw_{CY}, \phi^* J)$ is K\"ahler, Calabi-Yau.  Now fix $\gw'$ such that $\brs{\gw - \gw'}_{C^{\infty}(\gw)} < \ge$.  For sufficiently small $\ge > 0$, the one-parameter family $\gw_t = t \gw' + (1-t) \gw$ consists of symplectic forms, and we deform $\phi^* J$ along this path to produce an almost complex structure $J'$ compatible with $\gw'$ such that $\brs{\phi^* J - J'}_{C^{\infty}(\gw)} < C(\gw) \ge$ (cf. \cite{SCF} Lemma 4.3).  For $\ge$ chosen sufficiently small at the outset, the pair $(\gw', J')$ satisfies the hypothesis of Theorem \ref{t:dynstabSCF} relative to the Calabi-Yau structure $(\gw, \phi^* J)$, and thus the solution to symplectic curvature flow with initial condition $(\gw', J')$ exists globally and converges to a Calabi-Yau structure $(\gw_{\infty}, J_{\infty})$, which further satisfies $\brs{\gw - \gw_{\infty}}_{C^{\infty}(\gw)} < C(\gw) \ge$.  By construction, since $J'$ is connected by a smooth path to $J$, it follows that $c_1(M, J') = 0$, and this property is preserved along the symplectic curvature flow.  This in turn implies that the cohomology class of $[\gw']$ is preserved along the flow, thus $[\gw_{\infty}] = [\gw']$.  Since both $\gw_{\infty}$ and $\gw'$ are $\ge$-close to the symplectic form $\gw$, it follows that the path $t \gw' + (1-t) \gw_{\infty}$ consists of cohomologous symplectic forms, again applying Moser's Lemma we obtain a diffeomorphism $\psi$ such that $\psi^* \gw_{\infty} = \gw'$.  It follows that the pair $(\gw', \psi^* J_{\infty})$ is K\"ahler, in fact Calabi-Yau, finishing the proof.
\end{proof}


\end{document}